\documentclass[12pt,reqno]{amsart}
\usepackage{graphicx,amsmath,amsthm}
\usepackage{bbm}
\usepackage{hyperref}
\hypersetup{colorlinks}
\pdfoutput=1

\newcommand{\old}[1]{}

\newtheorem{theorem}{Theorem}

\newtheorem{lemma}[theorem]{Lemma}
\newtheorem{corollary}[theorem]{Corollary}

\theoremstyle{remark}

\theoremstyle{definition}

\DeclareSymbolFont{AMSb}{U}{msb}{m}{n}
\DeclareMathSymbol{\C}{\mathbin}{AMSb}{"43} 
\DeclareMathSymbol{\EE}{\mathbin}{AMSb}{"45} 
\DeclareMathSymbol{\N}{\mathbin}{AMSb}{"4E} 
\DeclareMathSymbol{\PP}{\mathbin}{AMSb}{"50} 
\DeclareMathSymbol{\Q}{\mathbin}{AMSb}{"51} 
\DeclareMathSymbol{\R}{\mathbin}{AMSb}{"52} 
\DeclareMathSymbol{\Z}{\mathbin}{AMSb}{"5A}

\begin{document}

\title[The looping constant of $\Z^d$]{The looping constant of $\Z^d$}

\author{Lionel Levine}
\address{Department of Mathematics, Cornell University, Ithaca, NY 14853. \url{http://www.math.cornell.edu/~levine}}

\author{Yuval Peres}
\address{Theory Group, Microsoft Research, 1 Microsoft Way, Redmond, WA 98052. \url{http://research.microsoft.com/en-us/um/people/peres/}}

\date{July 16, 2012}
\thanks{The first author was supported by NSF MSPRF 0803064 and NSF DMS 1105960.}
 
\begin{abstract}
The \emph{looping constant} $\xi(\Z^d)$ is the expected number of neighbors of the origin that lie on the infinite loop-erased random walk
%started at the origin
 in $\Z^d$.  Poghosyan, Priezzhev and Ruelle, and independently, Kenyon and Wilson, proved recently that $\xi(\Z^2)=\frac54$.
We consider the infinite volume limits as $G \uparrow \Z^d$ of three different statistics: (1) The expected length of the cycle in a uniform spanning unicycle of $G$; (2) The expected density of a uniform recurrent state of the abelian sandpile model on $G$; and (3) The ratio of the number of spanning unicycles of $G$ to the number of rooted spanning trees of $G$. 
%(a quantity which can be expressed in terms of a derivative of the Tutte polynomial of $G$ evaluated at $(1,1)$). 
  We show that all three limits are rational functions of the looping constant $\xi(\Z^d)$.  In the case of $\Z^2$ their respective values are 
$8$, $\frac{17}{8}$ and $\frac18$.
\end{abstract}

\maketitle

Fix an integer $d\geq 2$, and let $\xi=\xi(\Z^d)$ be the expected number of neighbors of the origin on the infinite loop-erased random walk in $\Z^d$ (defined in section~\ref{sec:limit}). 
We call this number $\xi$ the \emph{looping constant} of $\Z^d$. 
We explore several statistics of $\Z^d$ that can be expressed as rational functions of $\xi$.
Recently Poghosyan, Priezzhev, and Ruelle \cite{PPR11} and Kenyon and Wilson \cite{KW11} independently proved $\xi(\Z^2) = \frac54$, which implies that all of these statistics have rational values for $\Z^2$.

%The first quantity relates to graphs with a unique cycle. 
A \emph{unicycle} is a connected graph with the same number of edges as vertices.  Such a graph has exactly one cycle (Figure~\ref{fig:unicycle}).  If $G$ is a finite (multi)graph, a \emph{spanning subgraph} of $G$ is a graph containing all of the vertices of $G$ and a subset of the edges.  A \emph{uniform spanning unicycle} (USU) of $G$ is a spanning subgraph of $G$ which is a unicycle, selected uniformly at random.  

We regard the $d$-dimensional integer lattice $\Z^d$ as a graph with the usual nearest-neighbor adjacencies: $x \sim y$ if and only if $x-y \in \{ \pm \mathbf{e}_1,\ldots, \pm \mathbf{e}_d\}$, where the $\mathbf{e}_i$ are the standard basis vectors. 
An \emph{exhaustion} of $\Z^d$ is a sequence $V_1 \subset V_2 \subset \cdots$ of finite subsets such that $\bigcup_{n \geq 1} V_n = \Z^d$.  
%We denote by $o$ the origin in $\Z^d$ and assume that $V_1 = \{o\}$.  
Let $G_n$ be the multigraph obtained from $\Z^d$ by collapsing $V_n^c$ to a single vertex $s_n$, and removing self-loops at $s_n$.   We do not collapse edges, so $G_n$ may have edges of multiplicity greater than one incident to $s_n$.  Theorem~\ref{thm:main}, below, gives a numerical relationship between the looping constant $\xi$ and the expectation  
	\[ \lambda_n = \EE \left[ \mbox{length of the unique cycle in a USU of $G_n$} \right]. \]

The second statistic related to $\xi$ arises from the Bak-Tang-Wiesenfeld abelian sandpile model \cite{BTW87}.   A \emph{sandpile} on a finite subset $V$ of $\Z^d$ is a collection of indistinguishable particles on each element of $V$, represented by a function $\sigma: V \to \Z_{\geq 0}$ indicating how many particles are present at each $v \in V$.  A vertex containing at least $2d$ particles can \emph{topple} by sending one particle to each of its $2d$ neighbors; particles that exit $V$ are lost.  The sandpile is called \emph{stable} if no vertex can topple, that is, $\sigma(v) \leq 2d-1$ for all $v \in V$.  We define a Markov chain on the set of stable sandpiles as follows: at each time step, add one particle at a vertex of $V$ selected uniformly at random, and then perform all possible topplings until the sandpile is stable.  Dhar~\cite{Dhar90} showed that this stable configuration depends only on the initial configuration and not on the order of toppings, and that the stationary distribution of the Markov chain is uniform on the set of recurrent states.
%A \emph{recurrent sandpile} on $V_n$ is a function $\sigma : V_n \to \Z$ such that
%	\[ 0 \leq \sigma(x) \leq 2d-1. \]
%for all $x \in G_n$; and such that for all $A \subset G_n$ there exists $a \in A$ satisfying
%	\[ \sigma(a) \geq \mbox{deg}_A(a). \]

Theorem~\ref{thm:main} gives a numerical relationship between the looping constant $\xi$ and the expectation
	\[ \zeta_n = \EE \left[\mbox{number of particles at the origin in a stationary sandpile on $V_n$} \right]. \]
%	\[ \zeta_n = \EE\, \sigma(o), \]
%where $\sigma$ is a recurrent sandpile on $V_n$ chosen uniformly at random.

To define the last quantity of interest, recall that the \emph{Tutte polynomial} of a finite (multi)graph $G=(V,E)$ is the two-variable polynomial
	\[ T(x,y) = \sum_{A \subset E} (x-1)^{c(A)-1} (y-1)^{c(A) + \#A - n} \]
where $c(A)$ is the number of connected components of the spanning subgraph $(V,A)$.  Let $T_n(x,y)$ be the Tutte polynomial of $G_n$, and let
	\[ \tau_n = \frac{ \frac{\partial T_n}{\partial y}(1,1)  }{(\#V_n) T_n(1,1)}. \]
A combinatorial interpretation of $\tau_n$ is the number of spanning unicycles of $G_n$ divided by the number of rooted spanning trees of $G_n$.

For a finite set $V \subset \Z^d$, write $\partial V$ for the set of sites in $V^c$ adjacent to $V$.   		Denote the origin in $\Z^d$ by $o$.  Let us say that $V_1 \subset V_2 \subset \ldots $ is a \emph{standard exhaustion} if $V_1 = \{o\}$, $\# V_n = n$, and $\# (\partial V_n) / n \to 0$ as $n \to \infty$.

\begin{theorem}
\label{thm:main}
For any standard exhaustion of $\Z^d$, the following limits exist:
	\[ \tau =  \lim_{n \to \infty} \tau_n, \qquad
	   \zeta = \lim_{n \to \infty} \zeta_n, \qquad
	   \lambda = \lim_{n \to \infty} \lambda_n. \]
Their values are given in terms of the looping constant $\xi = \xi(\Z^d)$ by
	\begin{equation} \label{eq:theformulas} \tau = \frac{\xi-1}{2}, \qquad \zeta = d + \frac{\xi-1}{2}, \qquad \lambda = \frac{2d-2}{\xi-1}. \end{equation}
\end{theorem}

The case of the square grid $\Z^2$ is of particular interest because the quantities $\xi, \tau,\zeta,\lambda$ rather mysteriously come out to be rational numbers.

\begin{corollary} 
\label{conj:5/4}
In the case $d=2$, we have \cite{KW11,PPR11}
	\[ \xi = \frac54 \quad \mbox{and} \quad \zeta = \frac{17}{8}. \]
Hence by Theorem~\ref{thm:main},
	\[ \tau = \frac18 \quad \mbox{and} \quad \lambda = 8. \]
\end{corollary}

\begin{figure}
\centering
\includegraphics[height=0.35\textheight]{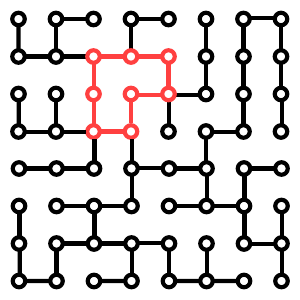} 
\caption{A spanning unicycle of the $8\times 8$ square grid.  The unique cycle is shown in red.}
\label{fig:unicycle}
\end{figure}

The value $\zeta(\Z^2) = \frac{17}{8}$ was conjectured by Grassberger in the mid-1990's \cite{Dhar06}.  Poghosyan and Priezzhev \cite{PP11} observed the equivalence of this conjecture with $\xi(\Z^2)=\frac54$, and shortly thereafter the proofs \cite{PPR11,KW11} of these two equalities appeared.  Our main contribution is showing that $\tau$, $\zeta$ and $\lambda$ are rational functions of $\xi$ in all dimensions $d \geq 2$, and consequently, the calculation of $\tau(\Z^2)$ and $\lambda(\Z^2)$.  As we will see, the proof of Theorem~\ref{thm:main} is readily assembled from known ingredients in the literature. However, we think there is some value in collecting these ingredients in one place.

%However, we do not know whether $\xi(\Z^d)$ is a rational number for any $d>2$.

%Jeng, Piroux and Ruelle~\cite{JPR} reduced the computation of $\zeta(\Z^2)$ to evaluation of a certain multiple integral~$J_2$, which they evaluated numerically as $0.5 \pm 10^{-12}$.  
%If this integral is exactly $\frac12$, then it would follow that $\zeta(\Z^2) = \frac{17}{8}$.  
%
%By Theorem~\ref{thm:main}, it suffices to compute any one of $\xi, \tau, \zeta, \lambda$ to compute all four quantities.  Thus, proving $J_2 = \frac12$ would suffice to verify Conjecture~\ref{conj:5/4}.  Recently, 
The proof that $\zeta(\Z^2) = \frac{17}{8}$ in \cite{KW11} uses Kenyon's theory of vector bundle Laplacians~\cite{Kenyon}, while the proof in \cite{PPR11} uses monomer-dimer calculations.  Both proofs 
%avoid evaluating $J_2$ directly, and both calculations 
involve powers of $1/\pi$ which ultimately cancel out.  For $i=0,1,2,3$ let $p_i$ be the probability that a uniform recurrent sandpile in~$\Z^2$ has exactly~$i$ grains of sand at the origin.  The proof of the distribution
\begin{align*}
p_0 &= \frac{2}{\pi^2} - \frac{4}{\pi^3} \\
p_1 &= \frac14 - \frac{1}{2\pi} - \frac{3}{\pi^2} + \frac{12}{\pi^3} \\
p_2 &= \frac38 + \frac{1}{\pi} - \frac{12}{\pi^3} \\
p_3 &= \frac38 - \frac{1}{2\pi} + \frac{1}{\pi^2} + \frac{4}{\pi^3}
\end{align*}
\old{
Majumdar and Dhar~\cite{MD91} showed that
	\[ p_0 = \frac{2}{\pi^2} - \frac{4}{\pi^3}. \] 
Jeng, Piroux and Ruelle~\cite{JPR}, building on work of Priezzhev~\cite{Priezzhev}, showed that $p_1$ and $p_2$ are related by
	\[ (\pi - 8)p_1 + 2(\pi -2)p_2 = \pi - 2 - \frac{3}{\pi} + \frac{12}{\pi^2} - \frac{48}{\pi^3}. \]
Since $p_0 + p_1 + p_2 + p_3 = 1$, only one more relation is needed to determine the distribution of $\sigma(o)$.  
}
is completed in \cite{PPR11,KW11}, following work of \cite{MD91,Priezzhev,JPR}.  In particular, $\zeta(\Z^2) = p_1 + 2p_2 + 3p_3 = \frac{17}{8}$.
%As to why $\zeta = p_1 + 2p_2 + 3p_3$ should be an exact rational number $17/8$, Jeng et al.\ write, ``The striking simplicity of this result clearly calls for a better explanation than just long calculations.''

The objects we study on finite subgraphs of $\Z^d$ also have ``infinite-volume limits'' defined on $\Z^d$ itself: Lawler~\cite{Lawler80} defined the infinite loop-erased random walk, Pemantle~\cite{Pem91} defined the uniform spanning tree in~$\Z^d$, and Athreya and J\'{a}rai~\cite{AJ} defined the infinite-volume stationary measure for sandpiles in~$\Z^d$.  As for the Tutte polynomial, the limit
	\[ t(x,y) := \lim_{n \to \infty} \frac1n \log T_n(x,y) \] 
can be expressed in terms of the pressure of the Fortuin-Kasteleyn random cluster model with parameters $p=1-\frac{1}{y}$ and $q=(x-1)(y-1)$.  By a theorem of Grimmett (see \cite[Theorem 4.58]{Grimmett}) this limit exists for all real $x,y>1$.  Theorem~\ref{thm:main} concerns the behavior of this limit as $(x,y) \to (1,1)$; indeed, another expression for $\tau_n$ is
	\[ \tau_n =  \frac{\partial}{\partial y} \left.\left[\frac1n \log T_n(x,y) \right] \right|_{x=y=1}. \]
%	\[ \tau = \lim_{n \to \infty}  \frac{\partial}{\partial y} \left.\left[\frac1n \log T_n(x,y) \right] \right|_{x=y=1}. \]

Sections~\ref{sec:limit},~\ref{sec:burning} and~\ref{sec:tutte} collect the results from the literature that we will use.  These are 1.\ The existence of the infinite loop-erased random walk measure, due to Lawler; 2.\ Wilson's algorithm relating loop-erased random walk to paths in the wired uniform spanning forest;  3. the fact that each component of the wired uniform spanning forest has one end, proved by Benjamini, Lyons, Peres and Schramm;
4.\ the burning bijection of Majumdar and Dhar between spanning trees and recurrent sandpiles; 5.\ the theorem of Merino L\'{o}pez relating recurrent sandpiles to a specialization of the Tutte polynomial; and 6.\ a result of Athreya and J\'{a}rai which shows that in the infinite volume limit, the bulk average height of a recurrent sandpile coincides with the expected height at the origin.  Items 1-3 are used to prove existence of the limits in Theorem~\ref{thm:main}, and items 4-6 are used to relate them to one another.

\section{The looping constant of \texorpdfstring{$\Z^d$}{Zd}}
\label{sec:limit}

A \emph{walk} in $\Z^d$ is a sequence of vertices $\gamma_0,\gamma_1,\ldots$ such that $\gamma_i \sim \gamma_{i+1}$ for all $i$.  A \emph{path} is a walk whose vertices $\gamma_i$ are distinct.
Given a walk $\gamma = (\gamma_0,\ldots,\gamma_m)$, Lawler \cite{Lawler80} defined the \emph{loop-erasure} $LE(\gamma)$ as the path obtained by deleting the cycles of $\gamma$ in chronological order.  Formally,
	\[ LE(\gamma) = (\gamma_{s(0)}, \ldots, \gamma_{s(J)} ) \]
%	\[ LE(\gamma)_j = \gamma_{s(j)+1} \]
where $s(0)=0$ and for $j \geq 0$,
	\[ s(j+1) = 1+ \max \{ i \mid \gamma_i = \gamma_{s(j)} \}. \]
The sequence $s$ necessarily satisfies $s(J+1) = m+1$ for some $J \geq 0$.  This $J$ is the number of edges in the loop-erased path $LE(\gamma)$.

For $n\geq 1$, let $Q_n = [-n,n]^d \cap \Z^d$.  Let $\gamma^n$ be a simple random walk in $\Z^d$ stopped on first exiting the cube $Q_n$; that is,
	\[ \gamma^n = (\gamma_0, \ldots, \gamma_T ) \]
where $\gamma_0=o$; the increments $\gamma_{i+1}-\gamma_i$ are independent and uniformly distributed elements of $\{\pm \mathbf{e}_1,\ldots,\pm \mathbf{e}_d\}$; and $T = \min \{i \geq 0 | \gamma_i \notin Q_n\}$. 
Fix $k \geq 1$, and for $n \geq k$ let $P^n$ be the distribution of the first $k$ steps of the loop-erasure $LE(\gamma^n)$.  Thus $P^n$ is a probability measure on the set $\Gamma_k$ of paths of length $k$ starting at the origin in $\Z^d$. Lawler \cite[Prop.\ 7.4.2]{Lawler91} shows that the measures $P^n$ converge as $n \to \infty$ (this is easy in dimensions $d\geq 3$, but requires some work in dimension $d=2$ because of the recurrence of simple random walk).  Combining these measures on $\Gamma_k$ for $k \geq 1$, we obtain a measure $P$ on infinite paths starting at the origin in $\Z^d$.  

Let $\alpha$ be an infinite loop-erased random walk starting at the origin in $\Z^d$, that is, a random path with distribution $P$.  We define the \emph{looping constant} $\xi$ of $\Z^d$ as the expected number of neighbors of the origin lying on $\alpha$.  By symmetry, each of the $2d$ neighbors $\pm \mathbf{e}_i$ is equally likely to lie on $\alpha$, so
	\begin{equation}
	\label{eq:theloopingconstant}
	 \xi  = 2d P(\mathbf{e}_1 \in \alpha).
	 \end{equation}
	 
\begin{figure}
\centering
\includegraphics[height=0.35\textheight]{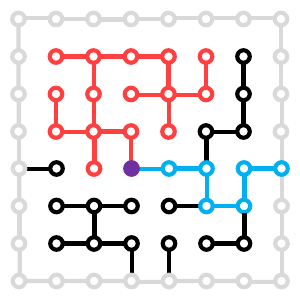} 
\caption{A wired spanning forest of the $6\times 6$ square grid.  The origin is the filled purple vertex.  Its path of ancestors is shown in blue, and its subtree of descendants is shown in red.  The wired boundary is shown in gray.  In this example, the north and west neighbors are descendants of the origin, but the south and east neighbors are not.}
\label{fig:desc}
\end{figure}

We first relate the looping constant $\xi$ to the wired uniform spanning forest.
A \emph{wired spanning forest} $F$ of $V$ is an acyclic spanning subgraph of $V \cup \partial V$ such that for every $y \in V$ there is a unique path in $F$ from $y$ to $\partial V$ (Figure~2).  The vertices on this path are called \emph{ancestors} of $y$ in $F$.  If $x$ is an ancestor of $y$, we write $x <_F y$, and we also say that $y$ is a \emph{descendant} of $x$. 

A \emph{wired uniform spanning forest} (WUSF) of $V$ is a wired spanning forest of $V$ selected uniformly at random.  Pemantle~\cite{Pem91} showed that the path from $x$ to $\partial V$ in the wired uniform spanning forest on $V$ has the same distribution as the loop-erasure of a simple random walk started at $x$ and stopped on first hitting $\partial V$.  

Pemantle also defined the WUSF on $\Z^d$ itself.  An important property of the WUSF on $\Z^d$ is that almost surely, each connected component has only one end \cite{Pem91,BLPS,LMS08}, which means that any two infinite paths in the same component differ in only finitely many edges.  It follows that 
%there is an acyclic orientation of the WUSF ``toward infinity,'' such that 
the origin has only finitely many descendants.  In other words, only finitely many vertices $v \in \Z^d$ have the property that the origin is on the unique infinite path in the WUSF starting at~$v$.

Let $V_1 \subset V_2 \subset \cdots$ be an exhaustion of $\Z^d$, and let $F_n$ be a wired uniform spanning forest of $V_n$.  Let $F$ be a wired uniform spanning forest of $\Z^d$.  Denote by $\PP_n$ and $\PP$ the distributions of $F_n$ and $F$ respectively.  The defining property of $\PP$ is that if $\mathcal{A}$ is a cylinder event (i.e., an event involving a fixed finite set of edges of $\Z^d$) then $\PP_n(\mathcal{A}) \to \PP(\mathcal{A})$ as $n \to \infty$.  Note that the event that $x$ is a descendant of $o$ is not a cylinder event, which is why the proof of the following lemma requires some care.

\begin{lemma}
\label{lem:finitelooping}
Let 
	\[ \xi_n = \sum_{x \sim o} \PP_n \{o <_{F_n} x \} \]
 be the expected number of neighbors of $o$ which are descendants of $o$ in $F_n$.  Then as $n \to \infty$
	\[ \xi_n \to \xi \]
where $\xi$ is the looping constant \eqref{eq:theloopingconstant}.
\end{lemma}

\begin{proof}
Let $\xi^* = \sum_{x \sim o} \PP(o <_F x)$ be the expected number of neighbors of $o$ which are descendants of $o$ in $F$, where $F$ is a WUSF on $\Z^d$.  We first show that $\xi_n \to \xi^*$.  For fixed $k$, let $\mathcal{A}_k$ be the event that some vertex outside $V_k$ is a descendant of $o$ in $F$, and let $\delta_k = \PP(\mathcal{A}_k)$. Since each connected component of $F$ has one end, we have $\delta_k \to 0$ as $k \to \infty$.  Since $\mathcal{A}_k$ is a cylinder event, $\PP_n(\mathcal{A}_k) \to \delta_k$ as $n \to \infty$.  

Fix a neighbor $x$ of the origin $o$, and let $\gamma$ (resp.\ $\gamma^n$) be the infinite path from $x$ in $F$ (resp.\ the path from $x$ to $\partial V_n$ in $F_n$).  Let $\gamma_k$ (resp.\ $\gamma_k^n$) be the initial segment of $\gamma$ (resp.\ $\gamma^n$) until it first exits $V_k$.  Then
	\[ \PP_n(o \in \gamma_k^n) \;\stackrel{n \to \infty}{\longrightarrow}\; \PP(o \in \gamma_k) \;\stackrel{k \to \infty}{\longrightarrow}\; \PP(o \in \gamma). \]
Moreover the event that $o$ lies on $\gamma^n$ but not on the initial segment $\gamma_k^n$, is contained in the event $\mathcal{A}_k$ that $o$ has a descendant in $F_n$ outside $V_k$.  Hence
	\[ \PP_n(o \in \gamma^n) - \PP_n(o \in \gamma_k^n) \leq \PP_n(\mathcal{A}_k) \longrightarrow \delta_k.  \]
Now fix $\epsilon>0$ and take $k$ large enough that $\delta_k < \epsilon$.  We have
	\begin{align*} |\PP_n(o \in \gamma^n) - \PP(o \in \gamma)| 
	&\leq |\PP_n(o \in \gamma^n) - \PP_n(o \in \gamma_k^n)| \;+ \\
	&\quad + |\PP_n(o \in \gamma_k^n) - \PP(o \in \gamma_k)| \;+ \\
	&\qquad + |\PP(o \in \gamma_k) - \PP(o \in \gamma)|.
	\end{align*}
The right side is $<3\epsilon$ for sufficiently large $n$.  Hence $\PP_n(o <_{F_n} x) \to \PP(o <_F x)$ as $n \to \infty$.  Summing over neighbors $x$ of $o$ we obtain $\xi_n \to \xi^*$ as $n \to \infty$.  
%By symmetry, $\xi_n = 2d \PP_n(o \in \gamma^n)$ and $\xi^* = 2d \PP(o \in \gamma)$, which shows that $\xi_n \to \xi^*$.

Next we show that $\xi^*=\xi$.  By translation invariance of the WUSF on $\Z^d$,
	\[ \xi^* = \sum_{x \sim o} \PP(o <_F x) = \sum_{x \sim o} \PP(-x <_F o). \]
That is, $\xi^*$ is also the expected number of neighbors of $o$ which are \emph{ancestors} of~$o$ in~$F$.  To complete the proof it suffices to show that the unique infinite path $\gamma$ in $F$ starting at $o$ (consisting of $o$ and all of its ancestors) has the same distribution as the infinite loop-erased random walk $\alpha$ in $\Z^d$ starting at~$o$.  In the transient case $d\geq 3$, this is immediate from Wilson's method rooted at infinity \cite[Theorem~5.1]{BLPS}.  In the recurrent case $d=2$, Wilson's method requires us to choose a finite root $r \in \Z^2$.  By \cite[Proposition~5.6]{BLPS}, the unique path $\gamma_{r}$ from $o$ to $r$ in $F$ has the same distribution as the loop-erasure of simple random walk started at $o$ and stopped on first hitting $r$.  Fix $k \geq 1$, and let $\mathcal{B}_k$ be the event that~$r$ has an ancestor in the square $Q_k = [-k,k]^2 \cap \Z^2$.  Since the WUSF in $\Z^2$ has only one end, by taking~$r$ sufficiently far from~$o$ we can ensure that $\PP(\mathcal{B}_k)$ is as small as desired.  Let $\gamma_{r}^k, \gamma^k, \alpha^k$ be the initial segments of the paths $\gamma_r, \gamma, \alpha$ stopped when they first exit $Q_k$.  On the event $\mathcal{B}_k^c$ we have $\gamma_r^k = \gamma^k$.  For fixed $k$, as $r \to \infty$ the distribution of $\gamma_r^k$ converges to the distribution of $\alpha^k$, which shows that $\gamma^k$ has the same distribution as~$\alpha^k$.  Since this holds for all $k$, we conclude that $\gamma$ has the same distribution as~$\alpha$, which completes the proof.
\end{proof}

\section{Burning bijection}
\label{sec:burning}

Majumdar and Dhar~\cite{MD92} used Dhar's burning algorithm~\cite{Dhar90} to give a bijection between recurrent sandpiles and spanning trees.  Priezzhev~\cite{Priezzhev} showed how the local statistics of sandpiles correspond to the (not completely local) statistics of spanning trees under the burning bijection.  For a given vertex $x$, the correspondence relates the number of particles at $x$ in the sandpile with the number of neighbors of $x$ which are descendants of $x$ in the spanning tree.  This correspondence is described in several other papers in the physics literature; see~\cite[pages 7-8]{JPR} and the references cited there.

Let $G=(V,E)$ be a finite connected undirected graph, with self-loops and multiple edges permitted.  Fix a vertex $s \in V$ called the sink, which is not permitted to topple.  For each vertex $x \neq s$, fix a total ordering $<_x$ of the set $E_x$ of edges incident to~$x$.  

Let $t$ be a spanning tree of $G$.  For each vertex $x\neq s$, write $e_t(x)$ for the first edge in the path from $x$ to $s$ in $t$, and let $\ell_t(x)$ be the number of edges in this path.  The burning bijection associates to $t$ the recurrent sandpile
	\[ \sigma_t(x) = \delta_x - 1 - a_t(x) - b_t(x) \]
where $\delta_x = \# E_x$ is the degree of $x$, and
	\[ a_t(x) = \# \{ (x,y) \in E_x \,|\, \ell_t(y) < \ell_t(x)-1 \}; \]
	\[ b_t(x) = \# \{ (x,y) \in E_x \,|\, \ell_t(y) = \ell_t(x)-1 \mbox{ and } (x,y) <_x e_t(x) \}.   \]

Now let $t$ be a uniform random spanning tree of $G$, so that $\sigma_t$ is a uniform random recurrent sandpile.  We say that $y$ is a \emph{descendant} of $x$ in $t$ if $x$ is on the path from $y$ to $s$ in $t$.  Let $D_x$ be the number of neighbors of~$x$ that are descendants of~$x$ in~$t$.
%The next lemma says that conditional on $D_x = j$, the number of particles at vertex $x$ is uniformly distributed on $\{j,j+1,\ldots,\delta_x-1\}$.
According to the next lemma, conditional on $D_x$, the distribution of $\sigma_t(x)$ is uniform on $\{D_x, D_x + 1, \ldots, \delta_x -1\}$.

\begin{lemma}
\label{conditionaluniformity}
For any vertex $x \neq s$, and any $0 \leq j \leq k \leq \delta_x-1$,
	\[ \PP(\sigma_t(x)=k \,|\, D_x=j) = \frac{1}{\delta_x-j}. \]
\end{lemma}

To prove this, we will show something slightly stronger, conditioning on the entire tree minus the edge $e_x(t)$.  Write $\mathrm{des}_x(t)$ for the subtree of descendants of $x$ in $t$, and $\mathrm{nondes}_x(t)$ for the subtree of non-descendants of $x$ in $t$.  The former subtree is rooted at $x$, and the latter is rooted at $s$.  Suppose that $F = F_x \cup F_s$ is a two-component spanning forest of $G$, with one component $F_x$ rooted at $x$ and the other component $F_s$ rooted at $s$.  Consider the set of trees
	\[ \mathcal{T}(F) = \{t \mid \mathrm{des}_x(t)=F_x, ~ \mathrm{nondes}_x(t)=F_s \}. \]
See Figure~\ref{fig:twocomp}.

\begin{figure}
\centering
\includegraphics[height=0.35\textheight]{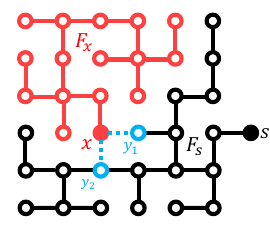} 
\caption{A two-component spanning forest $F = F_x \cup F_s$.  The set $\mathcal{T}(F)$ consists of the two trees $t_i = F \cup \{e_i\}$ obtained by adding one of the dashed edges $e_i = (x,y_i)$.  Since $\ell_{F_s}(y_1)=\ell_{F_s}(y_2)=5$, the ordering of the $e_i$ is determined by the ordering on $<_x$ on $E_x$.}
\label{fig:twocomp}
\end{figure}

\begin{lemma}
Let $j$ be the number of neighbors of $x$ that belong to $F_x$.  Then
	\begin{itemize}
	\item If $t \in \mathcal{T}(F)$, then $\sigma_t(x) \geq j$.
	\item For each integer $k \in \{j,j+1, \ldots, \delta_x-1\}$, there is
exactly one spanning tree $t \in \mathcal{T}(F)$ of $G$ with $\sigma_t(x)=k$.
	\end{itemize}
\end{lemma}

\begin{proof}
%Let $F = F_x \cup F_s$ be a two-component spanning forest of $G$, with one component $F_x$ rooted at $x$ and the other component $F_s$ rooted at $s$.  Conditional on $F \subset T$, the edge $e_x(T)$ is uniform on the set
Let
	\[ \mathcal{E} = \{ (x,y)\in E \,|\, y\in F_s \}. \]
Order the edges $e_1,\ldots,e_r$ of $\mathcal{E}$ so that
	\[ \ell_{F_s}(y_1) \leq \ldots \leq \ell_{F_s}(y_r) \]
where $e_i = (x,y_i)$, and $r=\delta_x-j$.  To break ties, if $\ell_{F_s}(y_i) = \ell_{F_s}(y_{i+1})$, then choose the ordering so that $e_i <_x e_{i+1}$. 

The elements of $\mathcal{T}(F)$ are precisely the trees 
	\[ t_i = F \cup \{e_i\} \]
for $i=1,\ldots,r$.  Note that $\ell_{t_i}|_{F_s} = \ell_{F_s}$, and
	\[ \ell_{t_i}(x) = \ell_{F_s}(y_i)+1. \]
Moreover, for any $z \in F_x$, since $x$ lies on the path from $z$ to $s$ in $t_i$, we have $\ell_{t_i}(z) > \ell_{t_i}(x)$.  By our choice of ordering of $\mathcal{E}$, we have
	\[ a_{t_i}(x) + b_{t_i}(x) = i-1 \]
hence 
	\[ \sigma_{t_i}(x) = \delta_x - 1-  a_{t_i}(x)-b_{t_i}(x) = \delta_x - i. \]
Thus as $i$ ranges from $1$ to $r$, the sandpile height $\sigma_{t_i}(x)$ assumes each integer value between $j = \delta_x - r$ and $\delta_x -1$ exactly once. \end{proof}

\section{Tutte polynomial}
\label{sec:tutte}

% some redundancies with PRE article

Let $G =(V,E)$ be any connected graph with $n$ vertices and $m$ edges (as before, $G$ is undirected and may have loops and multiple edges).  Fix a sink vertex $s \in V$, and let~$\delta$ be its degree.  We will use the following theorem of C.~Merino L\'{o}pez, which gives a generating function for recurrent sandpiles according to total number of particles.

\begin{theorem} \cite{Merino}
The Tutte polynomial $T_G(x,y)$ evaluated at $x=1$ is given by
	\[ T_G(1,y) = y^{\delta - m} \sum_{\sigma} y^{|\sigma|} \]
where the sum is over all recurrent sandpiles $\sigma$ on $G$ with sink at~$s$, and $|\sigma| = \sum_{x \neq s} \sigma(x)$ denotes the number of particles in $\sigma$.
\end{theorem}

Differentiating and evaluating at $y=1$, we obtain
	\[ \frac{\partial}{\partial y} T_G(1,1) = \sum_{\sigma} (\delta-m+|\sigma|). \]
The number of recurrent configurations is $T_G(1,1)$, so the mean height of a uniform random recurrent configuration is
	\[ \bar{\zeta}(G) := \frac{1}{n T_G(1,1)} \sum_{\sigma} |\sigma|. \]
Combining these expressions yields 
	
\begin{corollary}
\label{tuttepartial}
	\[ \bar{\zeta}(G) = \frac{1}{n} \left( m-\delta + \frac{\frac{\partial}{\partial y} T_G(1,1)}{T_G(1,1)}\right) = \frac{1}{n} \left( m-\delta + \frac{u(G)}{\kappa(G)} \right) \]
where $\kappa(G)$ is the number of spanning trees of $G$, and $u(G)$ is the number of spanning unicycles.
\end{corollary}

To derive the second equality of Corollary~\ref{tuttepartial} from the first, recall that 	
	\[ T_G(x,y) = \sum_{A \subset E} (x-1)^{c(A)-1} (y-1)^{c(A) + \#A - n} \]
where $c(A)$ is the number of connected components of the spanning subgraph $(V,A)$.  Here we interpret $0^0=1$.  Thus $T_G(1,1)$ counts connected spanning subgraphs of~$G$ containing exactly $n-1$ edges, which are precisely the spanning trees, while $\frac{\partial}{\partial y}T_G(1,1)$ counts connected spanning subgraphs containing exactly $n$ edges, which are precisely the unicycles.  The expression of $\bar{\zeta}$ in terms of the ratio $u(G)/\kappa(G)$ was exploited in \cite{FLW10} to estimate $\bar{\zeta}(G)$ when $G$ is a complete graph.

Next we relate the ratio $u(G)/\kappa(G)$ to the expected length of the cycle in a USU of~$G$.  If $U$ is a spanning unicycle of~$G$, then deleting any edge of its cycle yields a spanning tree, and each spanning tree is obtained from $m-n+1$ different unicycles in this way; hence
	\[ \sum_U |U| = (m-n+1) \kappa(G) \]
where $|U|$ denotes the length of the unique cycle of $U$, and the sum is over all spanning unicycles of~$G$.  Let 
	\[ \EE |U| = \frac{1}{u(G)} \sum_U |U| \]
be the expected length of the cycle in a uniform spanning unicycle of~$G$.  Then
	\begin{equation} \label{uoverkappa} \frac{u(G)}{\kappa(G)} = \frac{m-n+1}{\EE |U|}. \end{equation}
	
\section{Proof of Theorem~\texorpdfstring{\ref{thm:main}}{1}}

Let $V_1 \subset V_2 \subset \cdots$ be a standard exhaustion of $\Z^d$, and let $G_n$ be the multigraph obtained by collapsing $V_n^c$ to a single vertex $s_n$ and removing self-loops at $s_n$.
The last ingredient we need is a lemma saying that the bulk average sandpile height is the same as the expected height at the origin, in the large~$n$ limit.  This follows from the main result of Athreya and J\'{a}rai~\cite[Theorem~1]{AJ} on the infinite-volume limit of the stationary distribution of the abelian sandpile model.

\begin{lemma}
\label{lem:localglobal}
Let $\sigma_n$ be a uniform random recurrent sandpile on $G_n$.  Then
	\begin{equation} \label{eq:localglobal} \lim_{n \to \infty} \frac{1}{n } \sum_{x \in V_n} \EE \sigma_n(x) = \lim_{n \to \infty} \EE \sigma_n(o). \end{equation}
\end{lemma}

Let $D_n(o)$ be the number of neighbors of~$o$ which are descendants of~$o$ in the wired uniform spanning forest $F_n$ of $V_n$.  Note that wired spanning forests of $V_n$ are in bijection with spanning trees of $G_n$, so collapsing $\partial V_n$ to a single vertex converts $F_n$ into a uniform spanning tree of~$G_n$.  Let $\sigma_n(o)$ be the number of particles at the origin in a uniform recurrent sandpile on~$G_n$ with sink at~$s_n$.  For $k=0,1,\ldots,2d-1$ let
	\[ p_k = \PP(D_n(o) = k) \]
and
	\[ q_k = \PP(\sigma_n(o)=k). \]
By Lemma~\ref{conditionaluniformity} we have
	\[ q_k = \sum_{j=0}^{2d-1} \PP(\sigma_n(o)=k \,|\, D_n(o) = j) \PP(D_n(o) = j) = \sum_{j=0}^k \frac{p_j}{2d-j}. \]
In other words, the probabilities $p_k$ and $q_k$ are related by the linear system
	\begin{align*}
	q_0 &= \frac{p_0}{2d} \\
	q_1 &= \frac{p_0}{2d} + \frac{p_1}{2d-1} \\
	q_2 &= \frac{p_0}{2d} + \frac{p_1}{2d-1} + \frac{p_2}{2d-2} \\
	&\;\; \vdots \\
	q_{2d-1} &= \frac{p_0}{2d} + \frac{p_1}{2d-1} + \frac{p_2}{2d-2} + \cdots + \frac{p_{2d-2}}{2} + p_{2d-1}.
	\end{align*}
In particular, their expections are related by
	\begin{align*} \EE \sigma_n(o) = \sum_{k=1}^{2d-1} k q_k = \sum_{k=0}^{2d-1} k \sum_{j=0}^{k} \frac{p_j}{2d-j}. \end{align*}
Reversing the order of summation, we obtain 
	\begin{align*} \EE \sigma_n(o)	
%		&= \sum_{j=0}^{2d-1} \frac{p_j}{2d-j} \sum_{k=j}^{2d-1} k \\
		&= \sum_{j=0}^{2d-1} \frac{p_j}{2d-j} \left( \frac{ 2d(2d-1)}{2} - \frac{j (j-1)}{2} \right) \\
		&= \frac12 \sum_{j=0}^{2d-1} p_j (j+2d-1) \\
		&= \frac{\EE D_n(o) + 2d-1}{2}.
	\end{align*}	
Taking $n \to \infty$, we obtain by Lemma~\ref{lem:finitelooping}
	\[ \zeta = \frac{\xi + 2d -1}{2}. \]
	
Next we will use Corollary~\ref{tuttepartial} for the graph $G_n$ with sink at $s_n$.  Since $(V_n)_{n \geq 1}$ is a standard exhaustion, 
	\[ m-\delta = dn  - \#(\partial V_n) = (d-o(1)) n . \]  
Write $\bar{\zeta}_n := \bar{\zeta}(G_n)$.  By Lemma~\ref{lem:localglobal}, we have $\bar{\zeta}_n \to \zeta$ as $n \to \infty$, so Corollary~\ref{tuttepartial} gives
%By bringing the expectation outside the sum, the quantity inside the limit on the left side of (\ref{eq:localglobal}) equals $\bar{\zeta}(G_n)$, while the right side of~(\ref{eq:localglobal}) equals $\zeta_n$, so we have 
	\begin{align*} \lim_{n \to \infty} \frac{\frac{\partial T_n}{\partial y}(1,1)}{n  T_n(1,1)}  
	&= \lim_{n \to \infty} \left( \frac{\delta-m}{n } + \bar{\zeta}_n \right) \\
%	&= -\frac{b}{2} + \lim_{n\to \infty} \zeta_n \\
	&= - d + \zeta.
	\end{align*}
It follows that 
	\[ \tau = \frac{\xi - 1}{2}. \]

Finally, taking $G=G_n$ in (\ref{uoverkappa}), since $(m -n +1)/n  \to d-1$ as $n \to \infty$ we obtain
	\[ \lim_{n \to \infty} \EE |U_n| =  \lim_{N \to \infty} \frac{(m -n +1)\kappa(G_n)}{u(G_n)} = \lim_{n \to \infty} \frac{(d - 1)n  T_n(1,1)}{\frac{\partial T_n}{\partial y}(1,1)} = \frac{(d-1)}{\tau} \]
which gives 
	\[ \lambda = \frac{2d-2}{\xi-1}. \]
This completes the proof of Theorem~\ref{thm:main}.

\section{Concluding remarks}

A curious feature of Theorem~\ref{thm:main} is that the expected length of the cycle in a uniform spanning unicycle of $V_n \subset \Z^d$ remains bounded as $n\to \infty$.  This is not a special feature of subgraphs of $\Z^d$: on any finite graph, the expected length of the cycle in the USU can be bounded just in terms of the maximum degree and the maximum local girth.
Let $G$ be a finite graph of maximum degree $d$ in which every edge lies on a cycle of length at most $g$.  If $T$ is a uniform spanning tree of $G$ and $e$ is a random edge not in $T$, then $T \cup \{e\}$ is a random spanning unicycle of $G$; call this object a UST$^+$.  A spanning unicycle whose cycle has length $k$ can be written as $T \cup \{e\}$ for $k$ different pairs $(T,e)$.  Hence, the UST$^+$ is obtained from the USU via biasing by the length of the cycle: letting $q_k$ (respectively~$p_k$) be the probability that the cycle in the UST$^+$ (respectively USU) has length $k$, we have $q_k = kp_k / \sum_j j p_j$.  By Wilson's algorithm~\cite{Wilson}, the probability that the cycle in the UST$^+$ has length at most $g$ is at least $1/ (d(d-1)^{g-2})$.   
% first fix an edge (x,y).  There is a cycle C of length <=g through (x,y).  Now start Wilson's algorithm from x to build a spanning tree rooted at y.  With prob 1/d the first step is along the cycle in the other direction from y, and allowing for possible backtracks along the cycle we trace out the path C' = C-(x,y) with prob at least 1/(d(d-1)^{g-2}).  If we condition on (x,y) being the extra edge (i.e. not being present in the UST), this can only increase the probability that the UST contains C'.
Hence
	\[ \frac{\sum_{k \leq g} k p_k}{\sum_j j p_j} \geq \frac{1}{d(d-1)^{g-2}}. \]
The expected length of the cycle in the USU is  $\sum_j j p_j$, which is at most $gd(d-1)^{g-2}$.

It would be interesting to investigate the looping constant of graphs other than $\Z^d$.  Kenyon and Wilson \cite{KW11} calculated the looping constants of the triangular and hexagonal lattices and found that they too are rational numbers: $\frac53$ and $\frac{13}{12}$ respectively.  We expect that Theorem~\ref{thm:main} remains true if $\Z^d$ is replaced by any transitive amenable graph.  In this setting, the variable $d$ in equations \eqref{eq:theformulas} should be replaced by half the common degree of all vertices.

\section*{Acknowledgements}

We thank Richard Kenyon and David Wilson for helpful conversations.

\bibliographystyle{halpha}
\bibliography{fivefourths}

\end{document}